\documentclass[english,10pt]{article}
\usepackage[papersize={19cm,25cm},body={15cm,21cm},centering]{geometry}
\usepackage[T1]{fontenc}
\usepackage[latin9]{inputenc}
\usepackage{amsmath}
\usepackage{amssymb}
\usepackage{url}
\usepackage{babel}
 \usepackage{graphicx}
\usepackage{epsfig,subfigure,color}
\usepackage{graphics}
 \usepackage{graphicx}
\usepackage{amsmath,amsfonts,amsthm,latexsym}
\usepackage{mathrsfs,amsbsy}
\usepackage{multirow}
\usepackage{epstopdf}
\usepackage{algorithm}
\usepackage{algorithmic}
\usepackage{cite}

\newtheorem{theorem}{Theorem}

\newtheorem{remark}{Remark}
\newtheorem{proposition} {Proposition} 

\theoremstyle{definition}

\usepackage{bm}

 \usepackage{color}

\DeclareMathOperator{\tr}{trace}

\usepackage{enumitem} 
\begin{document}

\title{Hierarchical Alternating Least Squares Methods for Quaternion Nonnegative Matrix Factorizations}
\date{}
\author{
Junjun Pan 
\thanks{Department of Mathematics,
Hong Kong Baptist University.
Emails: junjpan@hkbu.edu.hk. J.Pan' s research is supported in part by Guang Dong Basic and Applied Basic Research Foundation (Grant No. 2023A1515010973).} 
}  
\maketitle

\begin{abstract}

In this paper,  we discuss a simple model for RGB color and polarization images under a unified framework of quaternion nonnegative matrix factorization (QNMF) and present hierarchical nonnegative least squares method to solve the factor matrices.
 The convergence analysis of the algorithm is discussed as well.  We test the proposed method in the polarization image and color facial image representation. Compared to the state-of-the-art methods, the experimental results demonstrate the effectiveness of the hierarchical nonnegative least squares method for QNMF model. 

\end{abstract}


\section{Introduction} 

Quaternions as an extension of complex numbers, were first introduced in 1843 by mathematician William Rowan Hamilton. They have gradually gained attentions and been widely used in the field of image and signal processing over the past thirty years. Quaternions are flexible to encode the components of 3D and 4D vector signals. For instance, they can be used to represent RGB color images \cite{le2003quaternion, Pei, pei1999color,sangwine1996fourier}, i.e., pixels of the red, green and blue channels are encoded in the three imaginary parts of quaternions. Another interesting application is using quaternions to process polarized signals \cite{flamant2020quaternion,gil2022polarized,le2004singular,pan2022separable}, i.e., the real component encodes the light intensity, while the imaginary part records the polarization state. We refer the reader to the recent survey \cite{miron2023quaternions} and the reference therein for more applications of quaternions.

Matrix factorization is an important linear dimensional reduction technique that can learn the latent structure of data, and has become a popular tool in signal and image processing community. In general, given a data matrix $M\in \mathbb{R}^{m\times n}$, matrix factorization aims to find two factor matrices $W\in \mathbb{R}^{m\times r}$, and $H\in \mathbb{R}^{r\times n}$ such that $M\approx WH$. Usually, the number $r$ is much less than $m$ and $n$. Matrix factorization plays a crucial role in low-rank approximation modeling for many problems, such as blind source signal separations. Similarly, researchers started to introduce matrix factorization into quaternion linear mixing model for quaternion signals.

In the blind polarized
source separation problems, the state of polarization can be entirely characterized by the four components of the
Stokes vector \cite{mcmaster1954polarization, berry1977measurement} and  reformulated in quaternion algebra \cite{kuntman2019quaternion}. Researchers in  \cite{le2004singular,javidi2011fast,via2010quaternion} considered low-rank approximation models based on quaternions, but these models lack physical constraints on the Stokes parameters. In 2020, Flamant et al. proposed a model named quaternion nonnegative matrix factorization (QNMF) in \cite{flamant2020quaternion}, which takes into account the physical constraints of the Stokes parameters.  For solving the QNMF model, they proposed a quaternion alternating least squares (QALS) algorithm. Like nonnegative matrix factorization, the QNMF is also not unique. To address this issue, Pan et al in \cite{pan2022separable} introduced the separability constraint into the QNMF model, and proposed a block coordinate descent algorithm to solve the model. 

For color image processing, the history of using quaternions matrix methods to process RGB color images can be traced back over twenty years. For example, \cite{le2003quaternion} proposed quaternion principal component analysis for color images. \cite{ke2023quasi} extended nonnegative matrix factorization model to quaternions and assumed that both factor matrices $W$ and $H$ are quternions with their imaginary part being nonnegative. Their model will lead to an embarrassing flaw in RGB color image processing because not only the real component of the resulted approximation will not vanish, but also the nonnegative constraints on the imaginary part of resulted approximation cannot be guaranteed. Hence, instead of computing the factor matrices of a given quaternion matrix, the researchers in \cite{lyu2024randomized,song2021low} considered the low rank approximation models for pure quaternions and applied the models onto color images. 

\subsection{Contributions and outline of this paper }

This paper will discuss a simple model for RGB color images and polarization images under the unified framework of quaternion nonnegative matrix factorization and provide a hierarchical algorithm accordingly. The contributions are summarized as follows.

1) We explore the flexibility of the QNMF model in processing RGB color images and show the feature images in the experiment.
2) We propose hierarchical nonnegative least squares for the unified QNMF model and analyze the convergence of the algorithm.
3) We test the proposed method in the polarization image  and color facial image representation. Compared to state-of-the-art techniques, the experimental results demonstrate the effectiveness of the hierarchical nonnegative least squares method.

The rest paper is organized as follows.
In Section 2, we briefly review quaternions,  and quaternion nonnegative matrix factorization model.  Section 3  proposes the quaternion  hierarchical nonnegative least squares to compute the factor matrices. Section 4 validates the effectiveness of the proposed method in numerical experiments on realistic polarization image and RGB color database. 
The concluding remarks are given in Section 5.

\section{Quaternion Matrix Factorization}

We will first introduce notations used throughout this paper.  The real number field, the complex number field and the quaternion algebra are denoted by $\mathbb{R}$, $\mathbb{C}$ and $\mathbb{H}$ respectively. Unless otherwise specified, the
lowercase letters represent real
numbers, the bold lowercase letters represent real vectors, and the bold capital letters  represent the real matrices, such as $a\in \mathbb{R}$, $\mathbf{a}\in \mathbb{R}^n$ and $\mathbf{A} \in \mathbb{R}^{m\times n}$. 
The numbers, vectors, and matrices under the quaternion field are represented by the corresponding symbols with breve, for example $ \breve{a} \in \mathbb{H}$,  $\mathbf{\breve{a}}\in \mathbb{H}^n$ and $\mathbf{\breve{A}}\in \mathbb{H}^{m\times n}$.

Generally, the quaternion field $\mathbb{H}$ is represented in the following Cartesian form,
$$ 
\breve{q}=q_0+\mathtt{i}q_{1}+\mathtt{j}q_2+\mathtt{k}q_3,
$$
where $q_0,q_1,q_2,q_3 \in \mathbb{R}$, and $\mathtt{i},\mathtt{j},\mathtt{k}$ are imaginary units such that 
 $$\mathtt{i}^2=\mathtt{j}^2=\mathtt{k}^2=-1,  \quad \mathtt{i}\mathtt{j}=-\mathtt{j}\mathtt{i}=\mathtt{k}, \quad \mathtt{j}\mathtt{k}=-\mathtt{k} \mathtt{j}= \mathtt{i},\quad \mathtt{k}\mathtt{i}=-\mathtt{i}\mathtt{k}=\mathtt{j},\quad \mathtt{i}\mathtt{j}\mathtt{k}=-1. $$
Any quaternion $\breve{q}$ can be simply written as $\breve{q}=\textit{Re}~ \breve{q}+ \textit{Im}~\breve{q}$ with real component $\textit{Re}~\breve{q}= q_0$ and imaginary component $\textit{Im}~\breve{q}= \mathtt{i}q_1+ \mathtt{j}q_2+\mathtt{k}q_3$. A quaternion $\breve{q}$ is called pure quaternion if its real component $\textit{Re}~\breve{q}=0$.  The quaternion conjugate $\bar{\breve{q}}$ and the modulus $|\breve{q}|$ of $\breve{q}$ are defined as 
$$\bar{\breve{q}}\doteq\textit{Re}~\breve{q} -\textit{Im}~\breve{q}=q_0-\mathtt{i} q_1-\mathtt{j}q_2-\mathtt{k}q_3,\quad  |\breve{q}|\doteq \sqrt{\breve{q}\bar{\breve{q}}}=\sqrt{q_0^2+q_1^2+q_2^2+q_3^2}.$$ 
A quaternion is a unit quaternion if its modulus equals to 1, i.e., $|\breve{q}|=1$. 
The dot product of two quaternions $\breve{a}=a_0+a_1 \tt{i}+a_2 \tt{j}+a_3\tt{k}$ and $\breve{b}=b_0+b_1 \tt{i}+b_2 \tt{j}+b_3\tt{k}$ is defined by
$\breve{a} \cdot \breve{b} = a_0b_0+a_1b_1+a_2b_2+a_3b_3$.
Similarly, for 
quaternion matrix $\breve{\mathbf{Q}}=(\breve{q}_{uv})\in \mathbb{H}^{m\times n}$, we denote its transpose $\breve{\mathbf{Q}}^T=(\breve{q}_{vu})\in \mathbb{H}^{n\times m}$ and its conjugate-transpose $ \breve{\mathbf{Q}}^*=(\bar{\breve{q}}_{vu})\in \mathbb{H}^{n\times m}$. We use $\bar{\breve{\mathbf{Q}}}=(\bar{\breve{q}}_{uv})\in \mathbb{H}^{m\times n}$ to denote the entry-wise conjugate of $\breve{\mathbf{Q}}$.
 
\subsection{Problem Statement}

Quaternion nonnegative matrix factorization (QNMF) was first proposed in \cite{flamant2020quaternion} to approximate polarization signals. Polarization describes the shapes and geometrical orientations of the oscillations, is one of the primary characteristics of transverse waves. The state of polarization can be characterized by the four components of the Stokes vector which are reformulated in quaternions. We take the polarization image as an example. Given a polarization image, each pixel contains light intensity and its polarization information, 
$$
\breve{m}_{st}= m^0_{st}+m^1_{st}\mathtt{i}+m^2_{st}\mathtt{j}+m^3_{st}\mathtt{k}\in \mathbb{H}_{S},
$$ 
where the set of nonnegative quaternions  $\mathbb{H}_S\subset \mathbb{H}$ is given 
\begin{equation}\label{Hs}
\mathbb{H}_S \doteq \{\breve{q}\in \mathbb{H}|\textit{Re}~q \geq 0 ~~ and ~~ |\textit{Im}~q|^2 \leq (\textit{Re}~q)^2\}.
\end{equation}
The constraints on $\mathbb{H}_S$ are due to the physical constraints of Stokes parameters. 

Given data matrix $\breve{\mathbf{M}}\in \mathbb{H}^{m\times n}_S$, quaternion nonnegative matrix factorization  finds $\breve{\mathbf{W}}\in \mathbb{H}^{m\times r}_S$ and $\mathbf{H}\in \mathbb{R}^{r\times n}_+$ such that
\begin{equation}\label{QNMF-WH}
\breve{\mathbf{M}}\approx \breve{\mathbf{W}}\mathbf{H},
\end{equation}
where $\breve{\mathbf{W}}$ and $\mathbf{H}$  are called the source matrix and the activation matrix respectively.  

On the other hand, it is known that for an RGB image, each pixel encodes three channels of Red, Green and Blue. Naturally, an RGB image can be represented as pure quaternions, that is, 
$\breve{\mathbf{M}}=\mathbf{M}_R \mathtt{i}+ \mathbf{M}_G \mathtt{j}+ \mathbf{M}_B \mathtt{k}\in \mathbb{H}^{m\times n},$
where $\mathbf{M}_R\in \mathbb{R}^{m\times n}_+$, $\mathbf{M}_G\in \mathbb{R}^{m\times n}_+$, $\mathbf{M}_B\in \mathbb{R}^{m\times n}_+$ are image matrices  in channels Red, Green and Blue respectively. Define 
\begin{equation}\label{Hn}
\mathring{\mathbb{H}}_{+}\doteq \{\breve{q}= q_1 \mathtt{i}+q_2\mathtt{j}+q_3\mathtt{k}\in \mathbb{H}~|~ q_1 \geq 0,~~ q_2\geq 0,~~ q_3\geq 0\},
\end{equation}
we have $\breve{\mathbf{M}}\in \mathring{\mathbb{H}}^{m\times n}_+$.
Similar to QNMF, we find source matrix $\breve{\mathbf{W}}\in \mathring{\mathbb{H}}^{m\times r}_+$, and the activation matrix $\mathbf{H}\in \mathbb{R}^{r\times n}_+$, such that $\breve{\mathbf{M}}\approx \breve{\mathbf{W}}\mathbf{H}$.

The above two problems are quaternion matrix factorization problems under different constrains. In the following, we consider the following  general problem: given quaternion matrix $\breve{\mathbf{M}}\in \mathbb{H}^{m\times n}_c$,  find $\breve{\mathbf{W}}\in \mathbb{H}^{m\times r}_c$, and $\mathbf{H}\in \mathbb{R}^{r\times n}_+$, such that
\begin{equation}\label{WH}
\min_{\breve{\mathbf{W}},\mathbf{H}} \|\breve{\mathbf{M}}-\breve{\mathbf{W}}\mathbf{H}\|,
\end{equation}
where 
\begin{itemize}
\item  $\mathbb{H}^{m\times r}_c\doteq \mathbb{H}^{m\times r}_S$ if the given  $\breve{\mathbf{M}}\in \mathbb{H}^{m\times n}_S$, for example, $\breve{\mathbf{M}}$ is polarization image matrix.

\item $\mathbb{H}^{m\times r}_c\doteq \mathring{\mathbb{H}}^{m\times r}_+$ if the given  $\breve{\mathbf{M}}\in\mathring{\mathbb{H}}^{m\times n}_+$, such as the RGB image.
\end{itemize}

Note that  every entry of the quaternion matrix $\breve{\mathbf{M}}$ is in the constrained quaternion set $\mathbb{H}_c$, it is necessary to show the plausibility of model, more precisely,

\begin{proposition} \label{Prop:1}
If $ \breve{\mathbf{W}}\in \mathbb{H}^{m\times r}_c$, $\mathbf{H} \in \mathbb{R}^{r\times n}_+$, $\breve{\mathbf{M}}=\breve{\mathbf{W}}\mathbf{H}$, then $\breve{\mathbf{M}}\in \mathbb{H}^{m\times n}_c$.
\end{proposition}
\begin{proof}
If  $\mathbb{H}_c= \mathring{\mathbb{H}}_+$, the results are obvious. 
If $\mathbb{H}_c=\mathbb{H}_S$, the result is given in \cite{flamant2020quaternion,pan2022separable}. 
\end{proof}

\section{Algorithms}

Like nonnegative matrix factorization, the solution $(\breve{\mathbf{W}},\mathbf{H})$ of QNMF is generally not unique under permutations and scaling. To solve the QNMF problem, the alternating least squares framework is taken into account in general. That is, at each iteration, one of the two factors is fixed and the other is updated such that the objective function \eqref{WH} is reduced.  The framework is given below. 

\begin{algorithm}[H]
\caption{ Quaternion ALS Framework \label{algo:qhnls}}
\begin{algorithmic}[1]
\REQUIRE Data matrix $\breve{\mathbf{M}} \in \mathbb{H}_c$, initial matrices $(\breve{\mathbf{W}}_0,\mathbf{H}_0) \in \mathbb{H}^{m\times r}_c\times \mathbb{R}^{r\times n}_+$. 

\ENSURE  Matrices $(\breve{\mathbf{W}},\mathbf{H})$   such that $ \breve{\mathbf{M}}\approx\breve{\mathbf{W}}\mathbf{H}$, where $(\breve{\mathbf{W}},\mathbf{H}) \in \mathbb{H}^{m\times r}_c\times \mathbb{R}^{r\times n}_+$ .

\FOR{$t = 1, 2,\cdots$}
\STATE  $\breve{\mathbf{W}}^{(t)}=\mathop{\arg\min}\limits_{\breve{\mathbf{W}}\in \mathbb{H}^{m\times r}_c}\|\breve{\mathbf{M}}-\breve{\mathbf{W}}\mathbf{H}^{(t-1)}\|^2_{F}$.

\STATE $\mathbf{H}^{(t)}=\mathop{\arg\min}\limits_{ \mathbf{H}\in \mathbb{R}^{r\times n}_+}\|\breve{\mathbf{M}}-\breve{\mathbf{W}}^{(t)}\mathbf{H}\|^2_{F}$.
 
\ENDFOR
\end{algorithmic}
\end{algorithm}
The least square solution of factor matrices $(\breve{\mathbf{W}},\mathbf{H})$ is given by
\begin{eqnarray}\label{ALS}
\breve{\mathbf{W}}^{(t)}&=& \mathbb{P}_{\mathbb{H}_c}[\mathbf{M}\mathbf{H}^{(t-1)T}\big(\mathbf{H}^{(t-1)}\mathbf{H}^{(t-1)T}\big)^{-1}];\label{ALS-W}\quad\quad
\\
\mathbf{H}^{(t)}&=&\mathbb{P}_{\mathbb{R}_+}\big[(\text{Re}[\breve{\mathbf{W}}^{(t)T}
\bar{\breve{\mathbf{W}}}^{(t)}])^{-1}
\text{Re}[\breve{\mathbf{W}}^{(t)T}
\bar{\breve{\mathbf{M}}}^{(t)}])
 \big]\label{ALS-H}.
\end{eqnarray}
If $\mathbb{H}_c=\mathbb{H}_S$, the iteration of $(\breve{\mathbf{W}},\mathbf{H})$ is given in \cite{flamant2020quaternion}. If  $\mathbb{H}_c=\mathring{\mathbb{H}}$, the iteration of $\breve{\mathbf{W}}$ is the similar but with different projection  onto  $\mathring{\mathbb{H}}$. 

We remark that the above iteration is simple and rough. Moreover, the iteration of $\mathbf{H}$ needs the projection onto $\mathbb{R}_+$ which may lead to many zeros in $\mathbf{H}$. This would cause the failure of the iterations.  Hence, we will consider hierarchical alternating least squares algorithms for factor matrices, and present them in Algorithm \ref{algo:qnls} and Algorithm \ref{algo:hnls} respectively.

\subsection{Algorithm for Solving $\breve{\mathbf{W}}$}

Let $F(\breve{\mathbf{W}},\mathbf{H})=\|\breve{\mathbf{M}}-\breve{\mathbf{W}}\mathbf{H}\|^2_F$, we have
\begin{eqnarray*}
F(\breve{w}_{tl})
= \sum^m_{t=1}\|\breve{\mathbf{M}}_{t,:}-\sum_p \breve{w}_{tp}\mathbf{H}_p\|^2_F = \sum^m_{t=1}\|\breve{\mathbf{M}}_{t,:}-\sum_{p\neq l}\breve{w}_{tp}\mathbf{H}_{p,:}-\breve{w}_{tl}\mathbf{H}_{l,:}\|^2_F.
\end{eqnarray*}
Then 
$$ 
\frac{\partial F(\breve{w}_{tl})}{\partial \breve{w}_{tl}}=-2(\breve{\mathbf{M}}_{t,:}-\sum_{p\neq l}\breve{w}_{tp}\mathbf{H}_{p,:}-\breve{w}_{tl}\mathbf{H}_{l,:})\mathbf{H}^T_{l,:}=0.
$$
That is
\begin{equation}\label{eq:source}
\breve{w}_{tl}=\frac{\breve{\mathbf{M}}_{t,:}\mathbf{H}^T_{l,:}-\sum\limits_{p\neq l}\breve{w}_{tp}\mathbf{H}_{p,:}\mathbf{H}^T_{l,:}}{\mathbf{H}_{l,:}\mathbf{H}^T_{l,:}}.
\end{equation}
Therefore, the $k$-th iteration of $\breve{\mathbf{W}}_{:,l}$ is given by
\begin{equation}
\breve{\mathbf{w}}^{(k)}_{:,l}=\frac{\breve{\mathbf{M}}\mathbf{H}^T_{l,:}-\sum\limits_{p<l}\mathbf{H}_{p,:}\mathbf{H}^T_{l,:}\breve{\mathbf{W}}^{(k)}_{:,p}-\sum\limits_{p> l}\mathbf{H}_{p,:}\mathbf{H}^T_{l,:}\breve{\mathbf{W}}^{(k-1)}_{:,p}}{\mathbf{H}_{l,:}\mathbf{H}^T_{l,:}}.
\end{equation}

Since the source matrices $\breve{\mathbf{W}}$ of different applications have their own properties, this requires us to project $\breve{\mathbf{W}}$ onto their specific set of constraints. Specifically,


\textbf{(1)} For the given polarization image matrix $\breve{\mathbf{M}}\in \mathbb{H}^{m\times n}_S$, the source matrix $\breve{\mathbf{W}}\in \mathbb{H}^{m\times r}_S$. To project $\breve{\mathbf{W}}$ onto $\mathbb{H}_S$, we note that the constraint set $\mathbb{H}_S$ in \eqref{Hs} is equivalent to the positive semi-definiteness of a $2 \times 2$ Hermitian matrix $\mathbf{J}$. 

More precisely, for any quaternion $\breve{q}=q_0+q_1 \tt{i}+ q_2\tt{j}+ q_3 \tt{k}$,  we write the Hermitian matrix 
$$
\mathbf{J}=\frac{1}{2}\left(
           \begin{array}{cc}
             q_0+q_2 & q_3+\mathbf{i}q_1  \\
            q_3-\mathbf{i}q_1   & q_0-q_2\\
           \end{array}
         \right)\in \mathbb{C}^{2\times 2}.
$$
Notice that $\tr{(\mathbf{J})}=q_0$, and $det{(\mathbf{J})}= q^2_0-q^2_1-q^2_2-q^2_3$. 
Then we deduce that \eqref{Hs} $\Leftrightarrow$ $\tr{(\mathbf{J})}\geq 0$ and $det{(\mathbf{J})} \geq 0$. In other words, $\mathbf{J}$ should be semi-positive definite. To guarantee $\mathbf{J}$  to be semi-positive definite, we project $\mathbf{J}$ onto the set of non-negative Hermitian matrices $\mathbb{Q}^{2\times 2}_+$, i.e., 
$$
\mathbb{P}_{\mathbb{Q}^{2\times 2}_+}(\mathbf{J})=\sum^2_{t=1}\max(0,\eta_t)\mathbf{u}_t\mathbf{u}^*_t\doteq 
\left(
           \begin{array}{cc}
             a & c \\
             \bar{c}   & b\\
           \end{array}
         \right),
$$
where $\eta_t$ and $\mathbf{u}_t$ are the $t$-th eigenvalue and eigenvector of the matrix $\mathbf{J}$. Now the projection of $\breve{q}$ onto $\mathbb{H}_S$ is given by
\begin{equation}\label{ProjHS}
\mathbb{P}_{\mathbb{H}_S}(\breve{q})=a+b+2\text{Im}(c)\tt{i}+(a-b)\tt{j}+2\text{Re}(c)\tt{k}.
\end{equation}
Therefore we obtain the $k$-th iteration of the source matrix $\breve{\mathbf{W}}$ in $\mathbb{H}_S$:
\begin{equation}\label{eq:W_polar}
\breve{\mathbf{W}}^{(k+1)}_{:,l}=\mathbb{P}_{\mathbb{H}_S}(\frac{\breve{\mathbf{M}}\mathbf{H}^T_{l,:}-\sum\limits_{p<l}\mathbf{H}_{p,:}\mathbf{H}^T_{l,:}\breve{\mathbf{W}}^{(k+1)}_{:,p}-\sum\limits_{p> l}\mathbf{H}_{p,:}\mathbf{H}^T_{l,:}\breve{\mathbf{W}}^{(k)}_{:,p}}{\mathbf{H}_{l,:}\mathbf{H}^T_{l,:}}),\quad l=1,\cdots,r.
\end{equation}

\textbf{(2)} For the given RGB image matrix $\breve{\mathbf{M}}\in \mathring{\mathbb{H}}_+$, the source matrix $\breve{\mathbf{W}}\in \mathring{\mathbb{H}}_+$, then the projection of $\breve{\mathbf{W}}$ onto $\mathring{\mathbb{H}}_+$ is given by:
for each entry $\breve{w}_{tl}=
w^R_{tl}\mathtt{i}+w^G_{tl}\mathtt{j}+w^B_{tl}\mathtt{k}$ 
$(t=1,\cdots,m;~l=1,\cdots,r)$, 
\begin{equation}\label{ProjH0}
\mathbb{P}_{\mathring{\mathbb{H}}_+}(\breve{w}_{tl})=\frac{1}{2}\big(\max\{\xi,|w^R_{tl}|+w^R_{tl}\}\mathtt{i}+\max\{\xi,|w^G_{tl}|+w^G_{tl})\}\mathtt{j}+\max\{\xi,|w^B_{tl}|+w^B_{tl}\}\mathtt{k}\big).
\end{equation}
where $\xi\lll 1$ is very small positive number to prevent the column of $\breve{\mathbf{W}}$ being zero. Hence for simplicity, we write the $k$-th iteration of the source matrix $\breve{\mathbf{W}}$ in $\mathring{\mathbb{H}}_+$:
\begin{equation}\label{eq:W_RGB}
\breve{\mathbf{W}}^{(k+1)}_{:,l}=\mathbb{P}_{\mathring{\mathbb{H}}_+}(\frac{\breve{\mathbf{M}}\mathbf{H}^T_{l,:}-\sum\limits_{p<l}\mathbf{H}_{p,:}\mathbf{H}^T_{l,:}\breve{\mathbf{W}}^{(k+1)}_{:,p}-\sum\limits_{p> l}\mathbf{H}_{p,:}\mathbf{H}^T_{l,:}\breve{\mathbf{W}}^{(k)}_{:,p}}{\mathbf{H}_{l,:}\mathbf{H}^T_{l,:}}),\quad l=1,\cdots,r.
\end{equation}

\begin{algorithm}[H]
\caption{ Hierarchical NLS-WQ (HNLS-WQ) \label{algo:qnls}}
\begin{algorithmic}[1]
\REQUIRE Matrix $\breve{\mathbf{M}} \in \mathbb{H}^{m\times n}_c$, and $\mathbf{H} \in \mathbb{R}^{r\times n}_+$. Initial matrix $\breve{\mathbf{W}}^{(0)}\in \Omega$, maximum iteration $iter$ and stopping criterion $\epsilon$. 

\ENSURE  Matrix $\breve{\mathbf{W}}$   such that $\min\limits_{\breve{\mathbf{W}}\in \mathbb{H}^{m\times r}_c}\|\breve{\mathbf{M}}-\breve{\mathbf{W}}\mathbf{H}\|^2_{F}$.

\STATE $\mathbf{A}=\mathbf{H}\mathbf{H}^T$, $\mathbf{B}=\breve{\mathbf{M}}\mathbf{H}^T$.

\WHILE{$k< iter$ or $\delta<\epsilon_0$}
\FOR{$l$ = 1 : r}
\STATE  $\mathbf{c}=\sum\limits^{l-1}_{t=1}a_{tl}\breve{\mathbf{W}}^{(k+1)}_{:,t}+\sum\limits^{r}_{t=l+1}a_{tl}\breve{\mathbf{W}}^{(k)}_{:,t}$;
\STATE $\breve{\mathbf{W}}^{(k+1)}_{:,l}= \frac{\mathbf{B}_{:,l}-\mathbf{c}}{a_{ll}}$;
\STATE \textbf{Projection onto $\Omega$}:
\begin{itemize}
\item If $\Omega=\mathbb{H}^m_{S}$: $\breve{\mathbf{W}}^{(k+1)}_{:,l}=\mathbb{P}_{\mathbb{H}_S}(\breve{\mathbf{W}}^{(k+1)}_{:,l})$.

\item If $\Omega=\mathring{\mathbb{H}}_+$: 
$\breve{\mathbf{W}}^{(k+1)}_{:,l}=\mathbb{P}_{\mathring{\mathbb{H}}_+}(\breve{\mathbf{W}}^{(k+1)}_{:,l})$
\end{itemize}

\ENDFOR
\STATE $\delta=\frac{\|\breve{\mathbf{W}}^{(k+1)}-\breve{\mathbf{W}}^{(k)}\|_F}{\|\breve{\mathbf{W}}^{(1)}-\breve{\mathbf{W}}^{(0)}\|_F}$.
\ENDWHILE
\end{algorithmic}
\end{algorithm}

\begin{remark}
If $\breve{\mathbf{M}}\in \mathbb{H}^{m\times n}_S$, HNLS-WQ needs $(2nr^2+8rmn)+k(4mr^2+26mr)$ flops in total; and if 
$\breve{\mathbf{M}}\in \mathring{\mathbb{H}}^{m\times n}_+$, HNLS-WQ needs $(2nr^2+6rmn)+k(3mr^2+6mr)$ flops in total; here $k$  refers to the number of iterations. The initial input for Algorithm \ref{algo:qnls} can be generated simply randomly. 
\end{remark} 

\subsection{Algorithm for Solving $\mathbf{H}$}

For solving $\mathbf{H}$, we rewrite  objective function \eqref{WH}, 
\begin{eqnarray*}
F(h_{lq})=\sum\limits^n_{q=1}\| \breve{\mathbf{M}}_{:,q}- \sum^r_{s\neq l}\breve{\mathbf{W}}_{:,s}h_{sq}-\breve{\mathbf{W}}_{:,l}h_{lq}\|^2_2,\\
\end{eqnarray*}
 where $h_{lq}$ stands for the $(l,q)$-th entry of matrix $\mathbf{H}$, with $l=1,2,\cdots, r$ and $q=1,2,\cdots,n$.

The gradient of  $F(h_{lq})$ is 
\begin{equation*}
\frac{\partial F(h_{lq})}{\partial h_{lq}}=-2 \breve{\mathbf{W}}^T_{:,l}\big( \breve{\mathbf{M}}_{:,q}-\sum^r_{s\neq l}\breve{\mathbf{W}}_{:,s}h_{sq}-\breve{\mathbf{W}}_{:,l}h_{lq}\big)=0.
\end{equation*}
Because of the nonnegative constrains on $\mathbf{H}$, $h_{lq}$ is obtained by 
\begin{equation}\label{hpj}
h_{lq}= \mathtt{P}_{\mathbb{R}_+}\Big(\frac{ \breve{\mathbf{W}}^T_{:,l}\breve{\mathbf{M}}_{:,q}- \sum\limits_{s\neq l}\breve{\mathbf{W}}^T_{:,l}\breve{\mathbf{W}}_{:,s}h_{sq}}{\breve{\mathbf{W}}^T_{:,l} \breve{\mathbf{W}}_{:,l}}\Big),\quad q=1,2,\cdots, n.
\end{equation}
Note that the row of $\mathbf{H}$ may equal to zero thanks to the projection $ \tt{P}_{\mathbb{R}_+}$, which will lead to failure of algorithms. To address this issue, we use a very small positive number $\xi$ to replace lower bound $0$ in the projection $ \tt{P}_{\mathbb{R}_+}$ in practice. That is, we define projection $ \tt{P}_{\mathbb{R}_{+\xi}}$:
$$
 \tt{P}_{\mathbb{R}_{+\xi}}(x)=\left\{\begin{array}{cl}
x, & \mbox{if} \ x\geq 0,\\
\xi, & \mbox{if} \  x<0,\\
\end{array}\right.
$$
where $\xi\lll 1$ is a very small positive number.
The row of $\mathbf{H}$ can be updated successively by 
\begin{equation}\label{Hrow_0}
\mathbf{H}^{(k+1)}_{l,:}= \tt{P}_{\mathbb{R}^n_{+\xi}}\Big(\mathbf{H}'^{(k+1)}_{l,:}\Big), 
\end{equation}
where
$$
 \mathbf{H}'^{(k+1)}_{l,:}=\frac{ \breve{\mathbf{W}}^T_{:,l}\breve{\mathbf{M}}- \sum\limits_{s< l}\breve{\mathbf{W}}^T_{:,l}\breve{\mathbf{W}}_{:,s}\mathbf{H}^{(k+1)}_{s,:}
 -\sum\limits_{s>l}\breve{\mathbf{W}}^T_{:,l}\breve{\mathbf{W}}_{:,s}\mathbf{H}^{(k)}_{s,:}}{\breve{\mathbf{W}}^T_{:,l} \breve{\mathbf{W}}_{:,l}}.
$$
Here $\mathbf{H}_{l,:}$ refers to the $l$-th row of $\mathbf{H}$. 
The iteration is presented in Algorithm \ref{algo:qhnls}, namely  Hierarchical NLS-HR. 
\begin{algorithm}[H]
\caption{ Hierarchical NLS-HR (HNLS-HR) \label{algo:hnls}}
\begin{algorithmic}[1]
\REQUIRE Matrix $\breve{\mathbf{M}} \in \mathbb{H}^{m\times n}_c$, and $\breve{\mathbf{W}} \in \mathbb{H}^{m\times r}_c$. Initial matrix $\mathbf{H}^{(0)}\in \mathbb{R}^{r\times n}_+$, maximum iteration $iter$ and stopping criterion $\epsilon$. 

\ENSURE  Matrix $\mathbf{H}$   such that $\min\limits_{\mathbf{H}\geq \xi}\|\breve{\mathbf{M}}-\breve{\mathbf{W}}\mathbf{H}\|^2_{F}$.

\STATE $\mathbf{A}= \breve{\mathbf{W}}^T\breve{\mathbf{W}} $, $\mathbf{B}=\breve{\mathbf{W}}^T\breve{\mathbf{M}} $.

\WHILE{$k< iter$ or $\delta<\epsilon_0$}
\FOR{$l$ = 1 : r}
\STATE  $\mathbf{c}=\sum\limits^{l-1}_{s=1}a_{ls}\mathbf{H}^{(k+1)}_{s,:}+\sum\limits^{r}_{s=l+1}a_{ls}\mathbf{H}^{(k)}_{s,:}$;
\STATE $\mathbf{H}^{(k+1)}_{l,:}= \frac{\mathbf{B}_{l,:}-\mathbf{c}}{a_{ll}}$;
\STATE \textbf{Projection onto $\mathbb{R}^n_{+\xi}$:} $\mathbf{H}^{(k+1)}_{l,:}=\max(\xi, \mathbf{H}^{(k+1)}_{l,:})$.
\ENDFOR
\STATE $\delta=\frac{\|\mathbf{H}^{(k+1)}-\mathbf{H}^{(k)}\|_F}{\|\mathbf{H}^{(1)}-\mathbf{H}^{(0)}\|_F}$.
\ENDWHILE
\end{algorithmic}
\end{algorithm}

\begin{remark}
 HNLS-HR needs $8(r^2m+rmn)+k(2nr^2+nr)$ flops in total, here $k$  refers to the number of iterations.  The initial input for Algorithm \ref{algo:hnls} can be generated simply randomly. 
\end{remark}  

\subsection{Discussion}

We remark that both HNLS-WQ and HNLS-HR are a block-coordinate descent method. For HNLS-WQ method for $\breve{\mathbf{W}}$, we note that the column $\breve{\mathbf{W}}_{:,l}$ is in closed set $\mathbb{H}^m_c$, where $\mathbb{H}^m_c$ can be $\mathbb{H}^m_S$ or $\mathring{\mathbb{H}}^m_+$; and the subproblems 
\begin{equation}\label{Wop}
\breve{\mathbf{W}}_{:,l}=\mathop{\arg\min}_{\breve{\mathbf{W}}_{:,l}\in \mathbb{H}^m_c}\|\breve{\mathbf{M}}-\sum_{s\neq l}\breve{\mathbf{W}}_{:,s}\mathbf{H}_{s,:}-\breve{\mathbf{W}}_{:,l}\mathbf{H}_{l,:}\|^2_F , \quad l=1,\cdots, r,
\end{equation}
are strictly convex. Here $\breve{\mathbf{W}}_{:,s}$  refers to the $s$-th column of $\breve{\mathbf{W}}$. 

For HNLS-HR method for $\mathbf{H}$, the row $\mathbf{H}_{l,:}$ is in closed convex set  $\mathbb{R}^{n}_{\xi}$, and the subproblems
\begin{equation}\label{Hop}
\mathbf{H}_{l,:}=\mathop{\arg\min}_{\mathbf{H}_{l,:}\geq \xi}\|\breve{\mathbf{M}}-\sum_{s\neq l}\breve{\mathbf{W}}_{:,s}\mathbf{H}_{s,:}-\breve{\mathbf{W}}_{:,l}\mathbf{H}_{l,:}\|^2_F , \quad l=1,\cdots, r,
\end{equation}
are strictly convex. Here $\mathbf{H}_{s,:}$ refers to the $s$-th row of $\mathbf{H}$. 
We can derive that the limit points of the iterates \eqref{eq:W_polar} for  $\breve{\mathbf{W}}$ in $\mathbb{H}_S$  (or \eqref{eq:W_RGB} for $\breve{\mathbf{W}}$ in $\mathring{\mathbb{H}}_+$ ) are stationary points;  and the limit points of iterates  \eqref{Hrow_0} are stationary points, based on the results in\cite{powell1973search,bertsekas1997nonlinear}. The convergence results are presented as follows.
\begin{theorem}\label{thm:convergence}
The points of algorithm \ref{algo:qnls} and algorithm \ref{algo:hnls} will converge to  stationary points respectively.
\end{theorem}
We remark that HNLS-WQ and HNLS-HR are inspired by hierarchical nonnegative least squares (HNLS) \cite{gillis2012accelerated} that is an exact block-coordinate descent method for solving  nonnegative matrix factorization. The convergence results (Theorem \ref{thm:convergence}) are similar to HNLS (see Theorem 3 in reference \cite{gillis2008nonnegative}).

\begin{remark}
To compute the factor matrices $\breve{\mathbf{W}}$ , we remark that one can employ  HNLS-WQ method or simply the least squares iterates \eqref{ALS-W}. Similarly, for solving $\mathbf{H}$, HNLS-HR or the least squares iterates \eqref{ALS-H} can be used. Therefore, it is necessary to investigate the effectiveness of these iterations. In Section \ref{sec:hyper}, we will consider the following methods.

\begin{itemize}
\item QHALS refers to the iterations that compute $\breve{\mathbf{W}}$ by HNLS-WQ and $\mathbf{H}$ by  HNLS-HR respectively. 

\item Qals-Rhals refers to the method that computes $\breve{\mathbf{W}}$ by iterations          \eqref{ALS-W} and $\mathbf{H}$ by  HNLS-HR respectively.

\item Qhals-Rals refers to the method that computes $\breve{\mathbf{W}}$ by HNLS-WQ and $\mathbf{H}$ by iterations \eqref{ALS-H} respectively. 
\end{itemize}
\end{remark}

\section{Numerical Experiments}\label{sec:hyper}

In this section, we test the proposed methods (QHALS, Qals-Rhals, Qhals-Rals) on  the  polarization images and RGB color facial images.  All experiments were run on Intel(R) Core(TM) i7-5500 CPU @2.20GHZ with 8GB of RAM using Matlab. We compare with the state-of-the-art method QALS for solving the QNMF model \eqref{WH}. Specifically,

\begin{enumerate}
\item If data matrix $\breve{\mathbf{M}}\in \mathbb{H}_S$, we apply the QALS (Quaternion alternating least squares algorithm \cite{flamant2020quaternion}) which is extended from standard ALS to solve the QNMF model.

\item  If data matrix $\breve{\mathbf{M}}\in 
\mathring{\mathbb{H}}_+$, we apply QALS framework  given in \cite{flamant2020quaternion} but with different projection on $\mathring{\mathbb{H}}_+$. 

\end{enumerate}

\begin{remark}
We tested QADMM \cite{sanchez2021automatic} which solves the QNMF model by using quaternion alternating direction method of multipliers. But it runs much slower than the others, and the results are similar with QALS, so we do not present it here.    
\end{remark}

In order to show the effect of these algorithms, we will report the following quality measures: the approximations of the reconstructed image of all the methods, which are defined below.
\begin{enumerate}
\item  The total relative approximation:
\begin{equation}\label{appro}
 \Upsilon=1-\dfrac{\|\breve{\mathbf{M}}-\breve{\mathbf{W}}\mathbf{H}\|_F}{\|\breve{\mathbf{M}}\|_F}.
\end{equation}
  
\item 
 The relative approximation on each component:
 \begin{equation} \label{appro1}
  \Upsilon_{l}= 1- \frac{ \|\mathcal{S}_l(\breve{\mathbf{M}})-\mathcal{S}_l(\breve{\mathbf{W}})\mathbf{H}\|_F}{\|\mathcal{S}_l(\breve{\mathbf{M}})\|_F}, \quad l=0,1,2,3.
\end{equation}
where $\mathcal{S}_l(\breve{\mathbf{M}})$ denotes the $l$-th component of the corresponding quaternion matrix.

\end{enumerate}

The stopping criterion of all the methods for solving QNMF model: 
We stop the methods when the following condition holds,
$$
\frac{e(k)-e(k-1)}{e(k-1)}\leq \delta,
\quad \text{where}\quad e(k)=\quad 
\frac{\|\breve{\mathbf{M}}-\breve{\mathbf{W}}^{(k)}\mathbf{H}^{(k)}\|_F}{\|\breve{\mathbf{M}}\|_F}.
$$
The matrices $(\breve{\mathbf{W}}^{(k)}, \mathbf{H}^{(k)})$ are the solution at iteration $k$. We use the threshold $\delta = 10^{-4}$ for all the data sets. The maximum number of iterations is 1000.  We set the maximum computation time $10^4$ seconds. 
%
\subsection{Polarization Image Representation}

To test the methods, we chose the polarization image "cover"  from polarization image dataset constructed by Qiu in \cite{qiu2021linear}. The image "cover" contains $1024\times 1024$ pixels, and each pixel encodes the information of intensity and polarization. We resized the polarization image to $64\times 64$:
$$
\breve{\mathbf{T}}= \mathbf{T}^0+\mathtt{i} \mathbf{T}^1+\mathtt{j} \mathbf{T}^2+\mathtt{k} \mathbf{T}^3\in \mathbb{H}^{64\times 64}_S.
$$
The component $ \mathbf{T}^0\in \mathbb{R}^{64\times 64}_+$ represents the total intensity, and $( \mathbf{T}^1, \mathbf{T}^2, \mathbf{T}^3)$ stands for its polarization. Fig.~\ref{polarimage} shows the image "cover" captured by a colour camera and polarization image sensor. Its top row shows the picture of "cover" captured by the colour camera with polarization filter at $0^{\circ}$, and the bottom row shows the total intensity $\mathbf{T}^0$ and normalized  polarization $(\mathbf{T}^1/\mathbf{T}^0, \mathbf{T}^2/\mathbf{T}^0)$ measured by four intensities with linear polarizers oriented at $0^{\circ}, 45^{\circ}, 90^{\circ},135^{\circ}$.  Note that $ \mathbf{T}^3=0$ here, we do not present it in Fig.~\ref{polarimage}. 

\begin{figure} 
\includegraphics[width= \textwidth]{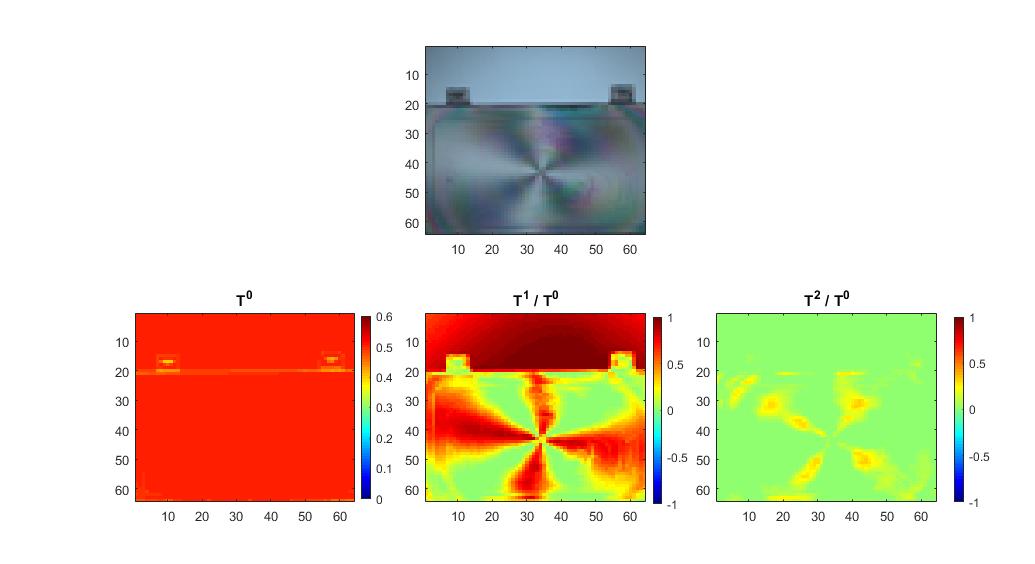}
\caption{Top: image "cover"; Bottom: the total intensity $\mathbf{T}^0$ and the normalized polarization $(\mathbf{T}^1/\mathbf{T}^0, \mathbf{T}^2/\mathbf{T}^0)$. }\label{polarimage}
\end{figure}

To better test the reconstruction effect of the methods for the image "cover", we do not apply the methods directly on $\breve{\mathbf{T}}$, instead, we divide image "cover" $\breve{\mathbf{T}}$ into $8\times 8$ small image blocks, and each block has $8\times 8$ pixels. The data matrix $\breve{\mathbf{M}}\in\mathbb{H}^{64\times 64}_S $ is then generated by these 64 blocks; that is, each column of $\breve{\mathbf{M}}$ is generated by vectorizing every image block. 

For all the methods, we apply the quaternion successive projection algorithm (QSPA) proposed in \cite{pan2022separable} to generate the initial guess. QSPA is a fast heuristic method that can generate good initial guesses efficiently.  Table  \ref{table:Po_result} reports the approximations of the reconstructed image of all the methods.  The best results are highlighted in bold, and the second best are highlighted in underline. From Table \ref{table:Po_result}, we observe that:

\begin{itemize}
\item In terms of the relative total approximation, and the approximations of all components, as the number of features $r$ increases,  the approximations of all the proposed methods increase, except the QALS for $r=16$. When $r=2,4,8$, the approximations obtained by QHALS are better than those by the other methods; and the Qals-Rhals gets the second best results in most cases. We notice that when $r=2$, the results from all the methods are similar. 
The reason may be that the initial results given by QSPA are sufficiently good to meet the stopping criteria. We remark that QHALS and Qhals-Rals all fail when $r=16$.

\item In terms of running time, the Qals-Rhals is the fastest method, andthe QALS is the second fast method. The QHALS and Qhals-Rhals are much slower than the other two methods when $r=8$, and fail when $r=16$. Both two methods involve the hierarchical iteration Algorithm \ref{algo:qnls} whose computational cost is high. 
 
\end{itemize}

The visualization of reconstructed images by $r=8$ features is shown in Fig.  \ref{polarresult}. The QHALS outperforms the other methods, as shown in the table. Considering both the time cost and approximation results, we conclude that the Qals-Rhals is the best among all the methods. 

\begin{figure}
\includegraphics[height=8.5cm,width=\textwidth]{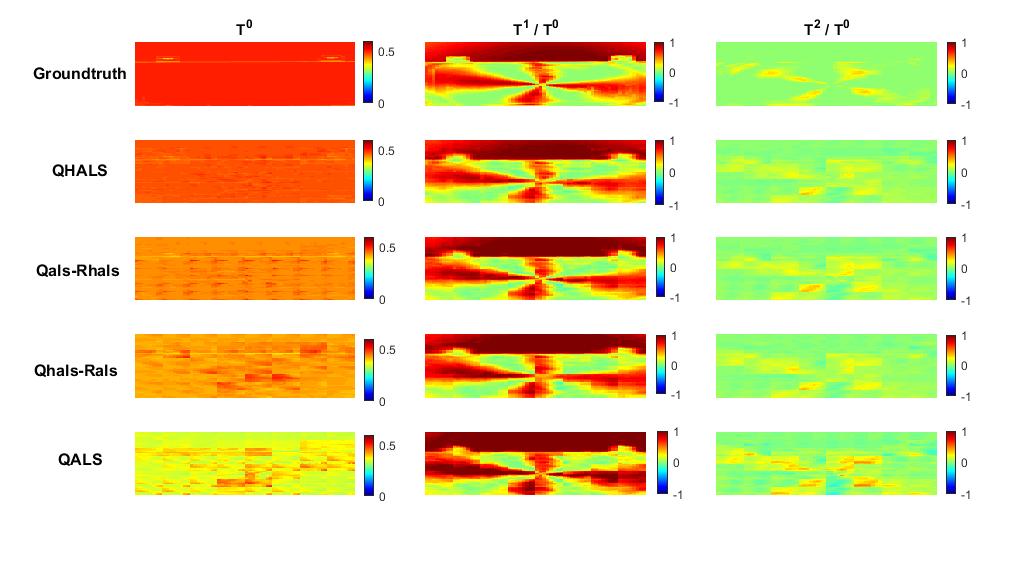}
\caption{Reconstruction results of 
 polarization image "cover'' $(r=8)$.}\label{polarresult}
\end{figure}

\begin{table} 
  \centering 
  \begin{tabular}{|c||c||c|c|c|c||c|}
  \hline
 ~& r&$\Upsilon$& $\Upsilon_0$ &  $\Upsilon_1$ &  $\Upsilon_2$ & Time(s)\\  
 \hline
 \multirow{3}*{QHALS}&2&\textbf{83.88} &\textbf{98.50}&\textbf{69.74}&\textbf{14.04}&1.80\\
  \cline{2-7}&4&\textbf{88.96} &\textbf{98.59}&\underline{80.40}&\textbf{19.02}&3.46\\
   \cline{2-7}&8&\textbf{93.47} &\textbf{98.61}&\textbf{89.75}& 35.62&6573.10\\
  \cline{2-7}&16&-- &--&--&--&--\\ 
  \hline
 \hline
 \multirow{3}*{Qals-Rhals}&2&\textbf{83.88} &\underline{98.47}&\underline{69.73}&\textbf{14.04}&\textbf{0.04}\\
   \cline{2-7}&4&\textbf{88.96} &\underline{98.58}&\underline{80.40}&\textbf{19.02}&\textbf{0.17}\\
    \cline{2-7}&8&\underline{93.07} &\underline{97.34}&\underline{89.42}&\underline{35.98}&\textbf{82.45}\\
    \cline{2-7}&16&\textbf{96.14} &\textbf{98.44}&\textbf{93.85}&\textbf{68.27}&\textbf{163.59}\\
 \hline
 \hline
  \multirow{3}*{Qhals-Rals}&2&\textbf{83.88} &\textbf{98.50}&\textbf{69.74}&\textbf{14.04}&1.77\\
   \cline{2-7}&4&87.90 &93.83&80.72&17.62&3363.00\\
    \cline{2-7}&8&92.15 &96.35&87.75&37.73&6735.20\\
     \cline{2-7}&16&-- &--&--&--&--\\
 \hline
 \hline
 \multirow{3}*{QALS}&2&\textbf{83.88} &\underline{98.47}&\underline{69.73}&\textbf{14.04}&\textbf{0.04}\\
   \cline{2-7}&4&87.20 &92.50&80.32&17.55&51.78\\
    \cline{2-7}&8&88.11&89.64&86.38&35.53&83.63\\
    \cline{2-7}&16&85.17&85.47&85.12&56.03&173.52\\
    \hline
 \end{tabular}
 \caption{Numerical results (in percent) for 
 polarization image "cover''.}\label{table:Po_result}
\end{table}

\subsection{RGB Color Image Representation}

In this experiment, we test the methods on a subset of the AR database \cite{wright2008robust} which consists 50 men and 50 women. Each subject has 26 pictures. The subset is constructed by selecting the first 4 images from the first 10 subjects. All the selected images are resized to $66\times 48$.  Fig. \ref{Figs.face40} displays all the samples of the subset. The data matrix $\breve{\mathbf{M}}$ is generated by these samples, and each column of $\breve{\mathbf{M}}$ is generated by vectorizing every image:
$$
\breve{\mathbf{M}}=\mathbf{M}^R \mathtt{i}+\mathbf{M}^{G}\mathtt{j}+\mathbf{M}^{B}\mathtt{k}\in \mathring{\mathbb{H}}^{3168\times 40}_+
$$
where the $\mathbf{M}^R$, $\mathbf{M}^G$ and $\mathbf{M}^B$ are image matrices encoding the values of the Red, Green and Blue channels respectively. 

We test all the methods on $\breve{\mathbf{M}}$. For all the methods, we  generate the initial guess simply by applying successive projection algorithm (SPA) \cite{gillis2013fast} on the stacking real matrix $[\mathbf{M}^R;\mathbf{M}^G;\mathbf{M}^B]$ to get the column set $\mathcal{K}$. The initial of activation matrix $\mathbf{H}$ is simply computed by HNLS method, and the initial of feature matrix $\breve{\mathbf{W}}$ is given by $\breve{\mathbf{M}}(:,\mathcal{K})$ accordingly.

\begin{figure}
\center
\includegraphics[width=0.75\textwidth]{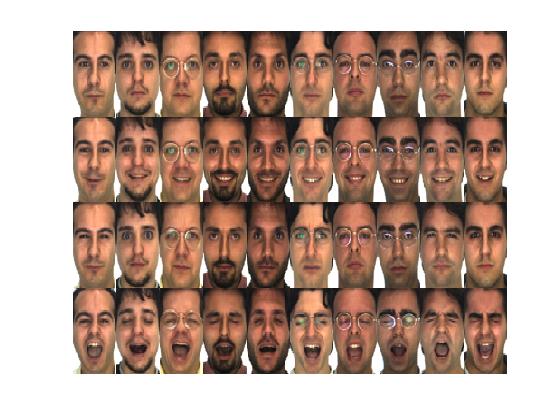}
\caption{The samples of the subset from the AR database}\label{Figs.face40}
\end{figure}

\begin{table}
  \centering 
  \begin{tabular}{|c||c||c|c|c|c||c|}
  \hline
 ~& r&$\Upsilon$& $\Upsilon_1$ &  $\Upsilon_2$ &  $\Upsilon_3$ & Time(s)\\  
 \hline
 \multirow{3}*{QHALS}&5&\underline{83.55} &\textbf{84.37}&\underline{82.82}&\underline{82.82}&\textbf{27.05}\\
  \cline{2-7}&10&\textbf{87.99}&\textbf{88.86}&\textbf{87.23}&\textbf{87.23}&\textbf{53.51}\\
  \cline{2-7}&15&\textbf{90.69}&\textbf{91.42}&\textbf{90.05}&\textbf{90.05}&288.54\\
   \cline{2-7}&20&\textbf{92.61}&\textbf{93.17}&\textbf{92.11}&\textbf{92.11}&523.13\\
   \cline{2-7}&25&\textbf{94.33}&\textbf{94.77}&\textbf{93.95}&\textbf{93.95}&576.60\\
  \hline
 \hline
 \multirow{3}*{Qals-Rhals}&5&83.51&84.32&82.78&82.78&189.51\\
   \cline{2-7}&10&\underline{87.88} &\underline{88.72}&\underline{87.13}&\underline{87.13}&205.12\\
   \cline{2-7}&15&\underline{90.48}&\underline{91.22}&\underline{89.84}&\underline{89.84}&196.00\\
   \cline{2-7}&20&\underline{92.24}&\underline{92.88}&\underline{91.69}&\underline{91.69}&219.65\\
   \cline{2-7}&25&\underline{94.12}&\underline{94.56}&\underline{93.73}&\underline{93.73}&213.14\\
 \hline
 \hline
  \multirow{3}*{Qhals-Rals}&5 &\textbf{83.56}&\underline{84.36}&\textbf{82.83}&\textbf{82.83}&88.01\\
   \cline{2-7}&10&87.36 &88.36&86.47&86.47&211.07\\
    \cline{2-7}&15&89.70&90.49&89.01&89.01&397.74\\
   \cline{2-7}&20&90.13&91.02&89.36&89.36&553.63\\
   \cline{2-7}&25&92.83&93.38&92.34&92.34&768.35\\
 \hline
 \hline
 \multirow{3}*{QALS}&5&82.89 &83.81&82.06&82.06&187.07\\
   \cline{2-7}&10&86.25 &87.06&85.53&55.53&197.94\\
   \cline{2-7}&15&81.15&81.93&80.45&80.45&214.33\\
   \cline{2-7}&20&82.00&83.26&80.89&80.89&188.59\\
   \cline{2-7}&25&74.73&76.41&73.23&73.23&175.48\\
    \hline
 \end{tabular}
 \caption{Numerical results (in percent) for 
 color facial images.}\label{table:RGB_result}
\end{table}

The approximations by the methods are presented in Table  \ref{table:RGB_result}. The best results are highlighted in bold, and the second best are highlighted in underline. From Table \ref{table:RGB_result}, we observe that: 

\begin{itemize}
\item In terms of the relative total approximation, and the approximations of all components, the approximations of all the proposed methods increase as the number of features $r$ increases, except for the QALS. The QHALS performs the best   for most cases. And the Qals-Rhals obtains the second best results. The approximations from the QALS are the least for all the cases, which means that QALS performs the worst.

\item In terms of running time, the QHALS is the fastest method when $r=5, 10$.  andthe QALS is the second fast method. The QHALS and Qhals-Rhals are slower than the other two methods when $r=15,20,25$.   
\end{itemize}

Fig. \ref{Figs.fea10} presents the $r=10$ features obtained by all the methods. The features from the best method QHALS and the second best method Qals-Rhals are similar. Fig. \ref{Figs.RGBapp10} shows the reconstructed images by their respective features. As we can see that, the QHALS outperforms all the other methods, which verify the results shown in Table \ref{table:RGB_result}. It is obvious that the QHALS is the best method among all the methods in terms of both the time cost and approximation results.

\begin{figure} 
\center
\includegraphics[width=0.8\textwidth]{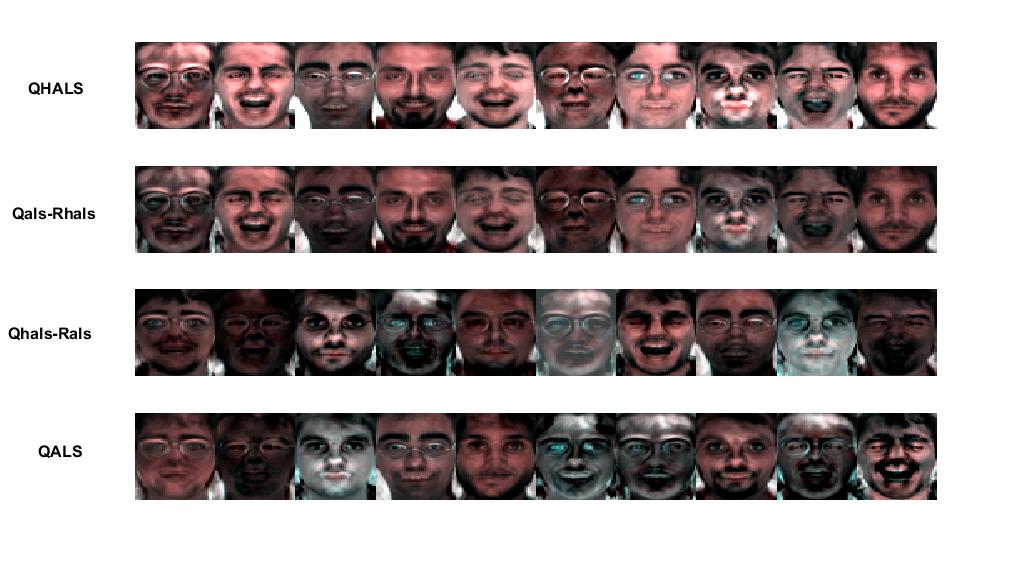}
\caption{The features of AR facial images by the methods when $r=10$.}\label{Figs.fea10}
\end{figure}

\begin{figure} 
\center
\includegraphics[width=\textwidth]{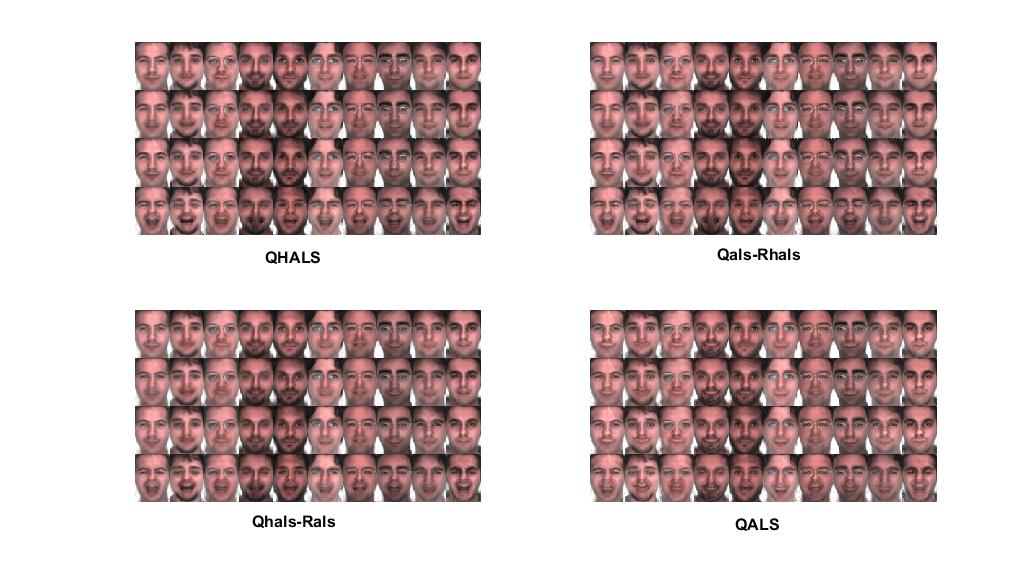}
\caption{The constructed approximation for AR facial images by the methods when $r=10$.}\label{Figs.RGBapp10}
\end{figure}

\section{Conclusion}

We propose a unified QNMF model for processing polarization images and RGB color images, and propose a hierarchical nonnegative least squares method to compute the factor matrices. We also discuss the convergence of the algorithm. To verify the effectiveness of the proposed method, experiments are conducted on polarization image and color facial image database. The experimental results demonstrate the effectiveness of the hierarchical nonnegative least squares method.

\end{document}